\documentclass[11pt]{amsart}
\usepackage{a4wide}
\usepackage{amsmath}
\usepackage{amssymb}      
\usepackage{graphicx}

\usepackage{mathabx}

\usepackage{amsopn}
\usepackage{epsfig}
\usepackage{amsfonts}
\usepackage{latexsym}
\usepackage{cancel}
\usepackage{calrsfs}
\usepackage{tikz}
\usepackage{enumerate}
\usepackage{dsfont}
\usepackage{cite}
\usepackage[colorlinks,linkcolor=blue,anchorcolor=blue,citecolor=blue]{hyperref}
\usepackage{eucal}
\DeclareFontFamily{U}{mathb}{}
\DeclareFontShape{U}{mathb}{m}{n}{
  <-5.5> mathb5
  <5.5-6.5> mathb6
  <6.5-7.5> mathb7
  <7.5-8.5> mathb8
  <8.5-9.5> mathb9
  <9.5-11> mathb10
  <11-> mathb12
}{}
\DeclareSymbolFont{mathb}{U}{mathb}{m}{n}
\DeclareMathSymbol{\righttoleftarrow}{\mathrel}{mathb}{"FD}
\DeclareMathSymbol{\mcurvearrowleft}{\mathrel}{mathb}{"E0}
\DeclareMathSymbol{\mcurvearrowleftB}{\mathrel}{mathb}{"E5}
\DeclareMathSymbol{\mcurvearrowleftC}{\mathrel}{mathb}{"EE}

\newtheorem{theorem}{Theorem}[section]
\newtheorem{lemma}[theorem]{Lemma}
\newtheorem{proposition}[theorem]{Proposition}

\theoremstyle{definition}

\theoremstyle{remark}
\newtheorem{remark}[theorem]{Remark}

\numberwithin{equation}{section}

\newcommand{\R}{\ensuremath{\mathbb{R}}}
\newcommand{\N}{\ensuremath{\mathbb{N}}}

\newcommand{\set}[1]{\left\{#1\right\}}

\newcommand{\f}{\infty}

\begin{document}

\title[ Invariant measures, matching properties and the associated subshifts]{Negative $\beta$-transformations: invariant measures, matching properties and associated subshifts}

\author[Y. Huang]{Yan Huang}
\address[Y. Huang]{School of Mathematics and Statistics, Wuhan University, Wuhan 430072, Hubei, People's Republic of China.}
\email{yanhuangyh@126.com}

\author[L. Liao]{Lingmin Liao}
\address[L. Liao]{School of Mathematics and Statistics, Wuhan University, Wuhan 430072, Hubei, People's Republic of China.}
\email{lmliao@whu.edu.cn}

\author[Y. Sun]{Yun Sun}
\address[Y. Sun]{Center for Mathematical Sciences, School of Mathematics and Statistics, Wuhan University of Technology, Wuhan 430070, Hubei, People's Republic of China.}
\email{sunyun@whut.edu.cn}

\begin{abstract}
For a real number $\beta>1$, we study the negative $\beta$-transformations, defined by $T_{-\beta}(x) := -\beta x + \lfloor\beta x\rfloor + 1$ for $x \in (0,1]$.
It is known that for each $\beta>1$, there exists a unique $T_{-\beta}$-invariant probability measure which is absolutely continuous with respect to the Lebesgue measure, denoted by $\nu_{-\beta}$. For two distinct non-integer real numbers $\beta_2>\beta_1>1$, we prove that $\nu_{-\beta_1}=\nu_{-\beta_2}$ if and only if $\beta_1$ is a root of the quadratic equation $x^2 - qx - p = 0$ (where $p, q \in \mathbb{N}$ and $p \leq q$) and $\beta_2 = \beta_1 + 1$.
Furthermore, we show that $T_{-\beta}$ satisfies the matching property for all generalized multinacci numbers $\beta$. Finally, we prove that the simple $(-\beta)$-numbers, i.e., the numbers whose corresponding $(-\beta)$-shifts are subshifts of finite type, are dense in the parameter space $(1, \infty)$.
\end{abstract}

\keywords{negative $\beta$-transformation, invariant measure, matching property, SFTs, simple $(-\beta)$ number, Yrrap number. }

\subjclass[2020]{11K55, 28D05, 37A10}

\maketitle

\section{Introduction}\label{sec:Introduction}
Let $\beta>1$ be a real number. As generalizations of the standard expansions in integer bases, R\'enyi \cite{Renyi-1957} and Parry \cite{Parry-1960} initiated the study of the so-called $\beta$-expansions by using the $\beta$-transformation $T_\beta: [0, 1) \to [0, 1)$  defined by
$$T_\beta(x) = \beta x - \lfloor \beta x \rfloor.$$
In fact, the $\beta$-expansion of $x \in [0, 1)$, generated by the successive iterations of $T_\beta$, is given by
$$x = \frac{\lfloor \beta x \rfloor}{\beta} + \frac{\lfloor \beta T_\beta(x) \rfloor}{\beta^2} + \frac{\lfloor \beta T_\beta^2(x) \rfloor}{\beta^3} + \cdots.$$
By convention, the digit sequence $(\lfloor \beta T_\beta^{n-1}(x) \rfloor)_{n\geq1}$ is also referred to as the $\beta$-expansion of $x$.
The dynamical and arithmetic properties of $T_\beta$ have been intensively studied for over half a century (see, e.g., \cite{Parry-1960,Blanchard-1989,Bertrand-Mathis-1986,Walters-1978,Frougny1992,Pfister-Sullivan-2005,Schmeling-1997,tan,Thompson-2012}).

\smallskip
Recently, as a natural analogue of the $\beta$-transformation, the $(-\beta)$-transformation (negative $\beta$-transformation) defined by
\begin{equation}\label{eq:def-beta}T_{-\beta}:(0,1] \to(0,1], \quad x \mapsto -\beta x + \lfloor \beta x \rfloor + 1,\end{equation}
has attracted a lot of attention from mathematicians.
The dynamical and arithmetic properties of $T_{-\beta}$ have been studied in  \cite{Ito-Sadahiro-2009, Liao-Steiner-2012, Steiner-2013}.
Analogously to the $\beta$-expansion, every $x\in(0,1]$ admits a $(-\beta)$-expansion of the from
\begin{equation}\label{eq:beta-expansion}x = \sum_{n=1}^\infty \frac{-d_n}{(-\beta)^n} = \frac{d_1}{\beta} - \frac{d_2}{\beta^2} + \frac{d_3}{\beta^3} - \cdots,
\end{equation}
where the digit sequences $d_1d_2\cdots$,  also referred to as the $(-\beta)$-expansion of $x$, is generated recursively via the iterates of $T_{-\beta}$:
\begin{equation}\label{eq:def-digit}
d_{n} = d_{-\beta,n}(x) := \lfloor \beta T_{-\beta}^{n-1}(x) \rfloor + 1\quad \forall \ n\geq 1.
\end{equation}
Here, $T_{-\beta}^0$ denotes the identity map.
 Examples of a $\beta$-transformation and a $(-\beta)$-transformation are depicted in Figure \ref{f:1}.

\begin{figure}[h!]\label{f:1}
  \centering
 \begin{tikzpicture}[scale=3]
                \draw[->] (0,0) node[below] {$0$}  --(1,0)node[below]{$1$}-- (1.1,0);
                \draw[->] (0,0) --(0,1)node[left] {$1$}-- (0,1.1) ;
                \draw[dashed] (1,0) -- (1,1) -- (0,1);
                \draw[thick, blue] (0,0) -- (0.618, 1);
                \draw[thick, blue] (0.618,0) -- (1, 0.618);
                \draw[dotted] (0.618,0) node[below]{\small $1/\beta$} -- (0.618,1);\draw[thick, blue, fill=white] (0.618, 1) circle (0.6pt); \draw[thick, blue, fill=white] (1,0.618) circle (0.6pt);
            \end{tikzpicture}\quad\quad \quad
            \begin{tikzpicture}[scale=3]
                \draw[->] (0,0) node[below] {$0$}  --(1,0)node[below]{$1$}-- (1.1,0);
                \draw[->] (0,0) --(0,1)node[left] {$1$}-- (0,1.1) ;
                \draw[dashed] (1,0) -- (1,1) -- (0,1);
                \draw[thick, red] (0,1) -- (0.618, 0);
                \draw[thick, red] (0.618,1) -- (1, 0.382);
                \draw[dotted] (0.618,0) node[below]{\small $1/\beta$} -- (0.618,1);\draw[thick, red, fill=white] (0.618, 0) circle (0.6pt);\draw[thick, red, fill=white] (0,1) circle (0.6pt);
            \end{tikzpicture}
  \caption{A $\beta$-transformation (left) and a $(-\beta)$-transformation (right), where $\beta=(1+\sqrt 5)/2$. }
\end{figure}

 In this paper, we explore three topics on the dynamics of $(-\beta)$-transformations: invariant measures, matching properties and associated subshifts.

\subsection{Invariant measures}

For any non-integer real number $\beta > 1$, it is well-known that the Lebesgue measure $\lambda$ is not $T_\beta$-invariant. However, R\'{e}nyi \cite{Renyi-1957} proved that there exists a unique $T_\beta$-invariant Borel probability measure $\nu_\beta$ which is equivalent to $\lambda$. Subsequently, Parry \cite{Parry-1964} provided an explicit formula for the density of $\nu_\beta$. The measure $\nu_\beta$ is now widely referred to as the \emph{Parry measure}. If $\beta$ is an integer, $\nu_\beta$ simply reduces to the Lebesgue measure.

\smallskip
Bertrand-Mathis \cite[Section III]{Bertrand-Mathis-1998} first observed that if $p,q\in\mathbb{N}$ with $p\leq q$, and
\begin{equation}\label{eq:coin-condition}
\beta_1>1 \text{ is a root of } x^2-qx-p=0,\quad \beta_2 = \beta_1 + 1,
\end{equation}
then, $\nu_{\beta_1} = \nu_{\beta_2}$. She also conjectured that condition \eqref{eq:coin-condition} is necessary for $\nu_{\beta_1} = \nu_{\beta_2}$ when $\beta_1, \beta_2$ are non-integer real numbers.
The conjecture of Bertrand-Mathis was solved by Brown and Yin \cite{Brown-Yin-1999}, who were unaware the conjecture.
In fact, the solution went unnoticed for a long time.
After presenting their seminal measure rigidity result, Hochman and Shmerkin \cite{Hochman-Shmerkin-2015} again raised the question of when $\nu_{\beta_1} = \nu_{\beta_2}$, which was subsequently answered by Huang and Wang \cite{Huang-Wang-2025}.

\smallskip
For $T_{-\beta}$, it was proved by Ito and Sadahiro \cite{Ito-Sadahiro-2009} (see also Liao and Steiner \cite{Liao-Steiner-2012}) that there exists a unique invariant  probability measure $\nu_{-\beta}$ which is absolutely continuous with respect to the Lebesgue measure (but not necessarily equivalent to the Lebesgue measure), for every $\beta>1$.
It is then natural to ask when $\nu_{-\beta_1} = \nu_{-\beta_2}$ for non-integer real numbers $\beta_1,\beta_2$.
Our first main result characterizes such coincidence of invariant measures for $T_{-\beta}$.
\begin{theorem}\label{thm:measure}
Let $\beta_2>\beta_1 > 1$ be two distinct non-integer real numbers. Then $\nu_{-\beta_1}=\nu_{-\beta_2}$ if and only if $\beta_1$ is a root of the quadratic equation $x^2 - qx - p = 0$ for some $p, q \in \mathbb{N}$ with $p \leq q$, and $\beta_2 = \beta_1 + 1$.
\end{theorem}

We remark that the negative $\beta$-transformations differ from their positive
counterparts: the density of $\nu_{\beta}$ is  positive on
$[0,1]$, while the density of
$\nu_{-\beta}$ can vanish on some subintervals of $[0,1]$
\cite{Liao-Steiner-2012}.
Moreover, for $1<\beta<2$, Kalle \cite{Kalle-2014} showed that $T_{\beta}$ and $T_{-\beta}$ are measurably isomorphic if and only if  $\beta$ is a multinacci number. Except for the golden mean, all parameters identified in Theorem \ref{thm:measure} lie outside this multinacci family. So, our analysis cannot be reduced to the positive case via the known measurable isomorphisms.
\smallskip

We also remark that addressing the coincidence of negative invariant  measures $\nu_{-\beta}$ requires a more challenging approach than the positive setting in \cite{Huang-Wang-2025}.
In the positive case, $0$ is a fixed point, and the normalizing constants of the densities for  $\nu_{\beta_1}$ and $\nu_{\beta_2}$ are equal if $\nu_{\beta_1}=\nu_{\beta_2}$.
However, for $T_{-\beta}$, $0$ is not a fixed point, and the ratio of the normalizing constants of the densities for $\nu_{-\beta_1}$ and $\nu_{-\beta_2}$ is $ {\beta_1(\beta_2+1)}/{\beta_2(\beta_1+1)}$.
This makes the analysis more delicate.
See details in Section \ref{sec:measure}.

\subsection{Matching properties}
In the proof of Theorem \ref{thm:measure}, we need to examine a phenomenon known as the \emph{matching property}, i.e., the property that there exists a finite integer $n\geq 1$ such that $T_{-\beta}^n(0)=T_{-\beta}^n(1)$.
 Here, we extend the definition of $T_{-\beta}$ to $0$ by setting $T_{-\beta}(0):=1$.

\smallskip
The matching property has attracted lots of  interest in the study of piecewise continuous maps, as it is connected to Markov partitions, entropy, and invariant measures. Notably, it has been widely investigated for $\alpha$-continued fraction maps \cite{Carminati, Carminati2013, Kraaikamp} and intermediate $\beta$-transformations \cite{Bruin, sun2023}. 
Motivated by these works, we study the matching property for $T_{-\beta}$.
For $q, m \in \mathbb{N}$, we define the \emph{generalized multinacci number} $\beta_{q,m} \in (q, q+1)$ as the unique real root of the equation
$$ x^m= q x^{m-1}+\cdots +qx+ q. $$
We show that for all generalized multinacci numbers, $T_{-\beta}$ satisfies the matching property.

\begin{theorem}\label{thm:matching}
If $\beta$ is a generalized multinacci number $\beta_{q,m}$, with $q,m\in\mathbb{N}$, then $T_{-\beta}$ satisfies the matching property.
\end{theorem}

\subsection{Associated subshifts}
For $\beta$-transformation, the closure (in the usual topology of symbolic space) of the set of all $\beta$-expansion of numbers in $[0,1)$, is denoted by $\mathcal{S}_{\beta}$. Note that the dynamical system $([0,1),T_{\beta})$ is topological conjugate to the \emph{$\beta$-shift} $(\mathcal{S}_{\beta}, \sigma)$, where $\sigma$ is the usual left-shift. 
 It is also wildly known that the dynamics of $([0,1),T_{\beta})$ and $(\mathcal{S}_{\beta}, \sigma)$ are entirely determined by the $\beta$-expansion of 1, i.e., $(\lfloor \beta T_\beta^{n-1}(1) \rfloor)_{n\geq1}$, where we extend the definition of $T_{\beta}$ to 1 by letting $T_{\beta}(1):=\beta-\lfloor \beta \rfloor$. Specifically, $(\mathcal{S}_{\beta}, \sigma)$  is a \emph{subshift of finite type} (SFT) if and only
if the $\beta$-expansion of 1 is finite ($\beta$ is called a simple $\beta$-number); and it is sofic if and only if
the $\beta$-expansion of 1 is eventually periodic ($\beta$ is called a $\beta$-number or Parry number).
Although both sets of simple $\beta$-numbers and  Parry numbers are at most countable, they are dense in $(1,\infty)$.

\smallskip
Analogously to the positive case, the closure (in the usual topology of symbolic space)  of the set of all
$(-\beta)$-expansions of numbers in $(0,1]$ is denoted by $\mathcal{S}_{-\beta}$. The dynamical system $((0,1],T_{-\beta})$ is topological conjugate to the \emph{$(-\beta)$-shift} $(\mathcal{S}_{-\beta},\sigma)$. 
The  $(-\beta)$-expansion of $1$  determines the dynamics of $(\mathcal{S}_{-\beta},\sigma)$, as summarized by the following proposition.

\begin{proposition}\label{lem:SFT-condition}
Let $\beta>1$ be a real number.
\begin{enumerate}
\item {\rm(\hspace{-0.07em}\cite[Theorem 12]{Ito-Sadahiro-2009})}
The $(-\beta)$-shift $(\mathcal{S}_{-\beta},\sigma)$ is sofic if and only if the $(-\beta)$-expansion of $1$ is eventually periodic.

\item {\rm(\hspace{-0.07em}{\cite[Proposition 3]{Frougny-Lai-2011}})}
The $(-\beta)$-shift $(\mathcal{S}_{-\beta},\sigma)$ is a subshift of finite type (SFT) if and only if the $(-\beta)$-expansion of $1$ is purely periodic.
\end{enumerate}
\end{proposition}

To align with Proposition \ref{lem:SFT-condition} on $\beta$-expansions, we call $\beta$ a \emph{simple $(-\beta)$-number} if the $(-\beta)$-expansion of $1$ is purely periodic, 
and following \cite{Liao-Steiner-2012}, we call $\beta$ is a \emph{Yrrap number} if the $(-\beta)$-expansion of $1$ is eventually periodic.
 Clearly, every simple $(-\beta)$-number is a Yrrap number. Since there are only countably many eventually periodic sequences, the set of Yrrap numbers (and hence simple $(-\beta)$-numbers) is countable. We prove that simple $(-\beta)$-numbers are dense; consequently, Yrrap numbers are dense as well.

\begin{theorem}\label{thm:sft-matching}
The set of simple $(-\beta)$-numbers is dense in $(1, +\infty)$.
\end{theorem}

We remark that for both positive and negative cases, the density property is established by constructing well-defined expansions that approximate the given expansion.
In the positive case, given $\beta$-expansion of $1$,  one can simply take its first $n$ terms to form a well-defined finite $\beta$-expansion that approximates the given one. In the negative setting, the $(-\beta)$-expansions of $1$ must satisfy four rules (see Proposition \ref{prop:beta-expansion-1}), and the order of $\mathcal{S}_{-\beta}$ defined in \eqref{eq:alter-order} depends on the parity of the first index at which two sequences differ. As a result, constructing purely periodic approximations becomes more delicate.

\smallskip
The rest of this paper is organized as follows. Section \ref{sec: preliminaries} reviews some basic properties of symbolic dynamics and $(-\beta)$-transformations. Section \ref{sec:measure} is devoted to the proof of Theorem \ref{thm:measure}. Section \ref{sec:4}  gives the proof of Theorem \ref{thm:matching} and discusses the connection between  the matching property and Yrrap numbers.
Finally, in Section \ref{sec:5} we study the density of simple $(-\beta)$ numbers  and provide the proof of Theorem \ref{thm:sft-matching}.

\section{Preliminaries on Symbolic Dynamics and $(-\beta)$-expansions}\label{sec: preliminaries}
The study of negative $\beta$-transformations, specifically its matching properties and the associated subshifts, relies on symbolic dynamics (see, e.g., \cite{Lind_Marcus_1995}). In this section, we provide the notation and concepts required for the subsequent analysis.

\subsection{Basics in symbolic dynamics}
For a given $\beta > 1$, let $\mathcal{A} = \{1, \cdots, \lfloor \beta \rfloor + 1\}$.
The symbolic space $\mathcal{A}^{\mathbb{N}}$ is defined as the set of all infinite sequences $(d_i) = d_1 d_2 \cdots$ where each $d_i \in \mathcal{A}$.
A \emph{word} is a finite string of digits from $\mathcal{A}$. We denote by $\mathcal{A}^*$ the set of all finite words, including the empty word $\epsilon$.
For a word $\mathbf{c} = c_1 \cdots c_m$, its length is denoted by $|\mathbf{c}| = m$.
For two words $\mathbf{c} = c_1 \cdots c_m$ and $\mathbf{d} = d_1 \cdots d_n$, their concatenation is written as $\mathbf{cd} = c_1 \cdots c_m d_1 \cdots d_n$.
Furthermore, for any $k \in \mathbb{N}$, $\mathbf{c}^k$ denotes the $k$-fold concatenation of $\mathbf{c}$ with itself. The notation  $\mathbf{cd}^\infty$ denotes the eventually periodic sequence formed by $\mathbf{c}$ and the infinite repetition of $\mathbf{d}$; it is purely periodic (periodic) if  $\mathbf{c}$ is the empty word $\epsilon$.
For a sequence $(x_n) \in \mathcal{A}^\mathbb{N}$ and a word $\mathbf{c} \in \mathcal{A}^*$, we say that $\mathbf{c}$ occurs in $(x_n)$ if there exist indices $i \le j$ such that $x_i \cdots x_j = \mathbf{c}$.
By convention, we say the empty word $\epsilon$ occurs in every sequence. For any $\mathcal{S} \subseteq \mathcal{A}^\mathbb{N}$, let $\mathcal{L}_n(\mathcal{S})$ denote the set of all words of length $n$ that occur in some $(x_n) \in \mathcal{S}$.
The language of $\mathcal{S}$ is defined as the collection,
$\mathcal{L}(\mathcal{S}) := \bigcup_{n=0}^\infty \mathcal{L}_n(\mathcal{S}).$

\smallskip
Throughout this paper, the sequences in $\mathcal{A}^{\mathbb{N}}$ are ordered according to the alternating lexicographical order, denoted by $\prec,\,\, \preceq,\,\, \succ$ or $\succeq$.
Specifically, let $(x_i) = x_1 x_2 \cdots$ and $(y_i) = y_1 y_2 \cdots$ be two distinct sequences, and let $k := \inf\{i \ge 1 : x_i \neq y_i\}$ be the first index at which  they differ.
Then $(x_i) \prec (y_i)$ if and only if
\begin{equation}\label{eq:alter-order}
\begin{cases}
x_k < y_k & \text{if } k \text{ is odd}, \\
x_k > y_k & \text{if } k \text{ is even}.
\end{cases}\end{equation}
We write $(x_i) \preceq (y_i)$ if $(x_i) \prec (y_i)$ or $(x_i) = (y_i)$.  The symbols $\succ$ and  $\succeq$ can be understood in obvious way.
The symbolic space $\mathcal{A}^{\mathbb{N}}$ is equipped with a usual  metric $\rho$, defined by
\begin{equation}\label{eq:rho}
\rho((x_i), (y_i)) = (\lfloor \beta \rfloor + 1)^{-\inf\set{n\ge 1: x_n\ne y_n}},
\end{equation}
under which $\mathcal{A}^{\mathbb{N}}$ is a compact metric space.

\smallskip
For a sequence $(x_n) = x_1 x_2 \cdots \in \mathcal{A}^\mathbb{N}$, the left-shift map
$\sigma: \mathcal{A}^\mathbb{N} \to \mathcal{A}^\mathbb{N}$ is defined by $\sigma(x_1 x_2 \cdots) = x_2 x_3 \cdots$. We call $(\mathcal{S},\sigma)$ a subshift if the subset $\mathcal{S} \subseteq \mathcal{A}^\mathbb{N}$ is closed  and shift-invariant, i.e., $\sigma(\mathcal{S}) \subseteq \mathcal{S}$.
Given a collection of forbidden words $\mathcal{F} \subset \mathcal{A}^*$, we denote by $\mathcal{S}_\mathcal{F}$ the set of all sequences in $\mathcal{A}^\mathbb{N}$ that do not contain any word from $\mathcal{F}$.
A subshift $(\mathcal{S},\sigma)$ is called a \emph{subshift of finite type (SFT)} if there exists a finite set $\mathcal{F} \subset \mathcal{A}^*$ such that $\mathcal{S} = \mathcal{S}_\mathcal{F}$.
  SFTs are fundamental objects in symbolic dynamics, and  play an essential role in dynamical systems.
The complexity of a subshift $(\mathcal{S},\sigma)$ is often measured by its topological entropy, defined by
\begin{equation}\label{eq:entropy}
h_{\rm{top}}((\mathcal{S},\sigma)) := \lim_{n \to \infty} \frac{1}{n} \log |\mathcal{L}_n(\mathcal{S})|.\end{equation}
Remark that the limit in \eqref{eq:entropy} exists since the sequence $\{\log|\mathcal{L}_n(\mathcal{S})|\}_{n\in\mathbb{N}}$ is sub-additive.

\smallskip
\subsection{$(-\beta)$-expansions and $(-\beta)$-shifts}
Recall from  \eqref{eq:beta-expansion} that for any real number $\beta > 1$,
each $x \in (0,1]$ can be represented as
$
x = \sum_{n=1}^\infty \frac{-d_n}{(-\beta)^n},
$
where the digits $d_n=d_{-\beta,n}(x)$ are determined by \eqref{eq:def-digit}.
The $(-\beta)$-expansion induces the coding map $d_{-\beta}$, which is
\begin{equation}\label{eq:def-pi}
d_{-\beta}: (0,1]\to \mathcal{S}_{-\beta};\quad  x \mapsto (d_n)_{n\geq1}.
\end{equation}
A sequence $d_1d_2\cdots$ is called \emph{$(-\beta)$-admissible} if there exists an $x\in(0,1]$ such that $d_{-\beta}(x)=d_{-\beta,1}(x)d_{-\beta,2}(x)\cdots=d_1 d_2\cdots$, i.e., lies in the image of $d_{-\beta}((0,1])$.

\smallskip
  Ito and Sadahiro \cite{Ito-Sadahiro-2009} characterized the admissibility for a conjugate transformation of $T_{-\beta}$. In fact,
  by using the topological conjugacy,
  their results can be translated to $T_{-\beta}$ as follows. 
  \begin{proposition}[\hspace{-0.07em}{\cite[Theorem 10]{Ito-Sadahiro-2009}}]\label{prop:1}
  A sequence $(d_n) \in \mathcal{A}^\mathbb{N}$ is $(-\beta)$-admissible if and only if
$d_{-\beta}^*(0) \prec \sigma^n(d_1 d_2 \cdots) \preceq d_{-\beta}(1) $ for all $n \ge 0$,
where $d_{-\beta}(1) = b_1 b_2 \cdots$ and the lower bound $d_{-\beta}^*(0)$ is given by
$$d_{-\beta}^*(0) =
\begin{cases}
(1 b_1 \cdots (b_p-1))^\infty & \text{if } d_{-\beta}(1) = (b_1 \cdots b_p)^\infty \text{ with } p \text{ odd}, \\
1 b_1 b_2 \cdots & \text{otherwise}.
\end{cases}$$
\end{proposition}

So, if $d_{-\beta}(1) = (b_1 \cdots b_p)^\infty$ with $p$ being odd, then,
\begin{equation}\label{eq:S-beta-odd}
\mathcal{S}_{-\beta} = \left\{ (d_n) \in \mathcal{A}^\mathbb{N} : (1 b_1 \cdots (b_p-1))^\infty \preceq \sigma^n(d_1 d_2 \cdots) \preceq d_{-\beta}(1) \text{ for all } n \ge 0 \right\}.\end{equation}
Otherwise, we have \begin{equation}\label{eq:S-beta-non-odd}
\mathcal{S}_{-\beta} = \left\{ (d_n) \in \mathcal{A}^\mathbb{N} : 1 d_{-\beta}(1) \preceq \sigma^n(d_1 d_2 \cdots) \preceq d_{-\beta}(1) \text{ for all } n \ge 0 \right\}.\end{equation}
The characterizations \eqref{eq:S-beta-odd} and \eqref{eq:S-beta-non-odd} establish the central role of $d_{-\beta}(1)$ in determining all $(-\beta)$-admissible expansions.
We also note that Proposition \ref{lem:SFT-condition} reduces the question of whether $(\mathcal{S}_{-\beta},\sigma)$ is a SFT or sofic to a purely number-theoretic question on $d_{-\beta}(1)$. So, understanding the combinatorial structure of $d_{-\beta}(1)$ is important. Steiner \cite{Steiner-2013} provided the following characterization.

   \begin{proposition}[\hspace{-0.07em}{\cite[Theorem 2]{Steiner-2013}}]\label{prop:beta-expansion-1}
   A sequence $(d_n) \in \mathcal{A}^\mathbb{N}$ is the $(-\beta)$-expansion of $1$ for some real number $\beta > 1$ if and only if it satisfies  the following conditions:
   \begin{enumerate}
   \item $\sigma^k(d_1 d_2 \cdots) \preceq d_1 d_2 \cdots$ for all $k \ge 1$;
   \item $d_1 d_2 \cdots \succ w_1 w_2 \cdots := 211222112\cdots$, where $w_1 w_2 \cdots$ is the sequence starting with $\phi^n(2)$ for all $n\geq  0$, with $\phi$ being the morphism of words defined by $\phi(2)=211$ and $\phi(1)=2$;

   \item $d_1 d_2 \cdots \notin \{d_1 \cdots d_k, \, d_1 \cdots d_{k-1}(d_k - 1)1\}^\infty \setminus \{(d_1 \cdots d_k)^\infty\}$ for all $k$ with $(d_1 \cdots d_k)^\infty \succ w_1 w_2 \cdots$;
   \item $d_1 d_2 \cdots \notin \{d_1 \cdots d_k 1, \, d_1 \cdots d_{k-1}(d_k+1)\}^\infty$ for all $k$ with $(d_1 \cdots d_{k-1}(d_k + 1))^\infty \succ w_1 w_2 \cdots$.\end{enumerate}
       \end{proposition}
Regarding the complexity of a dynamical system, the topological entropy $h_{\rm{top}}(T)$ of a piecewise continuous map $T$ can be defined via $(n, \epsilon)$-separated sets \cite{de1993}.
Takahashi \cite{Takahashi-1980} proved that the topological entropy of a piecewise continuous dynamical system  equals to its corresponding subshift (defined in \eqref{eq:entropy}). For $(-\beta)$-transformations, the equivalence of Takahashi \cite{Takahashi-1980}, implies that Proposition \ref{prop:entropy} (2) is a direct consequence of Proposition \ref{prop:entropy} (1). Proposition \ref{prop:entropy} will be used in the proof of Theorem~\ref{thm:sft-matching}.
\begin{proposition}\label{prop:entropy}
Let  $\beta > 1$ be a real number.
\begin{enumerate}
\item {\rm(\hspace{-0.07em}\cite[Proposition 3.7]{Shultz-2007})} If $T$ is a piecewise linear transformation on $[0,1]$ with slope $\pm\beta$, then $h_{\rm{top}}(T) = \log \beta$.
\item {\rm(\hspace{-0.07em}\cite[Proposition 5]{Frougny-Lai-2011})} The topological entropy of the $(-\beta)$-shift $(\mathcal{S}_{-\beta},\sigma)$ is equal to $\log \beta$.
\end{enumerate}\end{proposition}

\section{The coincidence of invariant measures} \label{sec:measure}
In this section, we study the coincidence of invariant measures and prove Theorem \ref{thm:measure}.
The invariant measure of $(-\beta)$-transformation was first studied by Ito and Sadahiro \cite{Ito-Sadahiro-2009}. In fact, they introduced the map
$$G_{-\beta} : \left[-\frac{\beta}{\beta+1}, \frac{1}{\beta+1}\right) \to \left[-\frac{\beta}{\beta+1}, \frac{1}{\beta+1}\right), \quad x \mapsto -\beta x - \left\lfloor \frac{\beta}{\beta+1} - \beta x \right\rfloor,$$
and proved that $G_{-\beta}$  possesses an absolutely continuous (with respect to Lebesgue measure) invariant measure $\mu_{-\beta}$ with density
$$g_{-\beta}(x) := \sum_{n \ge 0, \, x \ge G_{-\beta}^n (-\frac{\beta}{\beta+1})} \frac{1}{(-\beta)^n}, \quad \forall \ x \in \left[-\frac{\beta}{\beta+1}, \frac{1}{\beta+1}\right).$$

\smallskip
Later, Liao and Steiner \cite{Liao-Steiner-2012} observed  that $G_{-\beta}$ is conjugate to $T_{-\beta}$ via the map $$\psi(x) = \frac{1}{\beta+1} - x.$$
Consequently, for every $\beta > 1$, the transformation $T_{-\beta}$ has a unique absolutely continuous (with respect to Lebesgue measure)  invariant  probability measure $\nu_{-\beta}$ whose density is
\begin{equation}\label{eq:nor-density}
\widetilde{h}_{-\beta}(x) := \frac{h_{-\beta}(x)}{K_{-\beta}},
\end{equation}
where
\begin{equation}\label{eq:nonor-den}
h_{-\beta}(x) := \sum_{n \ge 0,\, T_{-\beta}^n (1) \ge x} \frac{1}{(-\beta)^n}=\sum_{n \ge 0} \frac{\mathds{1}_{(0,T_{-\beta}^n(x)]}(x)}{(-\beta)^n}, \quad x \in (0, 1],
\end{equation}
and $K_{-\beta}$ is a normalizing constant defined by
\begin{equation}\label{eq:nor-cons}
K_{-\beta} := \int_{(0,1]} h_{-\beta} \, d\lambda
= \sum_{n=0}^{\infty} \frac{T_{-\beta}^n (1)}{(-\beta)^n}.
\end{equation}

\smallskip
We first give a sufficient condition for the coincidence
 of invariant measures.
\begin{proposition}\label{prop:suff}
If $\beta_1>1$ is the positive root of the equation $x^2-qx-p=0$, where $p,q\in\mathbb{N}$ with $p\leq q$, and $\beta_2=\beta_1+1$, then we have $\nu_{-\beta_1}=\nu_{-\beta_2}$.
\end{proposition}
\begin{proof}
Note that $q< \beta_1 < q+1$, so we have $T_{-\beta_1} (1) = -\beta_1 + q+1$. By $\beta_1^2 - q \beta_1 - p =0$, a simple calculation implies that $T_{-\beta_1}^2 (1) =-\beta_1 + q+1$. Hence,
$T_{\beta_1}^n (1) =-\beta_1 + q+1$ for all $n \geq 1$.
Then, by \eqref{eq:nonor-den} and $\sum_{n=0}^\infty \frac{1}{(-\beta)^n}=\frac{\beta}{\beta+1}$, we have for any $x\in(0,1]$,
$${h}_{-\beta_1}(x) = \mathds{1}_{(0,1]}(x) +\sum_{n\geq1} \frac{\mathds{1}_{(0,-\beta_1+q+1]}(x) }{(-\beta)^n} = \frac{\beta_1}{\beta_1+1}\mathds{1}_{(0,-\beta_1+q+1]}(x) +\mathds{1}_{(-\beta_1+q+1,1]}(x),$$
where $\mathds{1}_{A}(x)$ is the indicator function of the set $A$.
By \eqref{eq:nor-cons}, we have  $$K_{-\beta_1} = \int_{(0,1]} h_{-\beta_1} \, d\lambda =\frac{\beta_1^2+p}{\beta_1(\beta_1+1)}.$$

\smallskip
Since $\beta_2=\beta_1+1$, a simple calculation implies that $T_{-\beta_2}(1)=-\beta_1 + q+1$ and $T_{-\beta_2}^2(1)=1$. Consequently,
$$T_{-\beta_2}^{2k+1} (1) = -\beta_1 + q+1\quad \text{and}\quad T_{-\beta_2}^{2k} (1) = 1\quad \text{for all }k\geq 0.$$
Again, by \eqref{eq:nonor-den}, for all $x\in(0,1]$, we have
\begin{align*}
{h}_{-\beta_2}(x) =&\sum_{n\geq 0}\frac{\mathds{1}_{(0,1]}(x)}{(-\beta_2)^{2n}} +
\sum_{n\geq0} \frac{\mathds{1}_{(0,-\beta_1+q+1]}(x) }{(-\beta_2)^{2n+1}}\\
=& \frac{\beta_2}{\beta_2+1}\mathds{1}_{(0,-\beta_1+q+1]}(x) +\frac{\beta_2^2}{\beta_2^2-1}\mathds{1}_{(-\beta_1+q+1,1]}(x)\\
=&\frac{\beta_2^2}{\beta_2^2-1}\left(\frac{\beta_2-1}{\beta_2}\mathds{1}_{(0,-\beta_1+q+1]}(x) +\mathds{1}_{(-\beta_1+q+1,1]}(x)\right)
=\frac{\beta_2^2}{\beta_2^2-1}{h}_{-\beta_1}(x),
\end{align*}
where the second equality holds by $\sum_{n=0}^\infty \frac{1}{(-\beta_2)^n}=\frac{\beta_2}{\beta_2+1}$ and $\sum_{n=0}^\infty \frac{1}{(-\beta_2)^{2n+1}}=\frac{\beta_2^2}{\beta_2^2-1}$.
Then, by \eqref{eq:nor-cons}, we have
$$K_{-\beta_2} =\frac{\beta_2^2}{\beta_2^2-1} K_{-\beta_1}.$$
Therefore, $\widetilde{h}_{-\beta_1}=\widetilde{h}_{-\beta_2}$, and hence $\nu_{-\beta_1} = \nu_{-\beta_2}$.
\end{proof}

 To prove the necessity of Theorem \ref{thm:measure}, we first give some properties of the density functions. 
\begin{proposition}\label{prop:property}
  Let $\beta > 1$ be a non-integer real number.
\begin{enumerate}
\item  ${h}_{-\beta}(x)$ is constant on  $(a,b)\subset(0,1)$ if and only if $(a,b) \cap \{ T_{-\beta}^n (1): n \geq 1 \}= \emptyset$.

\item $\displaystyle\lim_{x \to 0^{+}} {h}_{-\beta}(x) =\frac{\beta}{\beta+1}$.

     \smallskip
 \item Suppose there exists $k\geq 1$ such that  $T_{-\beta}^k(1)=1$.  Then $\lim_{x\to 1^-}h_{-\beta}(x)=\frac{\beta^m}{\beta^m-(-1)^m}$, where $m:=\min\{n\geq 1:T_{-\beta}^n(1)=1\}$.

     \smallskip
\item If  $T_{-\beta}^n(1)<1$ for all $n\geq 1$, then $\lim_{x\to 1^-}h_{-\beta}(x)=1$.
     \end{enumerate}
\end{proposition}
\begin{proof}
The proof of (1) is similar to that of Proposition 2.2 (iii) in \cite{Huang-Wang-2025};  therefore, we omit it here.
\smallskip

We now prove  (2).
  For $x \in (0,1]$, by \eqref{eq:nonor-den}, we have
  \begin{equation}\label{eq:h-n-x}
  {h}_{-\beta}(x) = \sum_{n \in N_x} \frac{1}{(-\beta)^n},\quad  \text{where   } N_x := \big\{ n \ge 0: T_{-\beta}^n (1) \geq  x \big\}.\end{equation}
For any $\varepsilon >0$, there exists $\kappa=\kappa(\varepsilon) \in \N$ such that
\begin{equation}\label{eq:h-e-x}\sum_{n=\kappa+1}^{\f} \left|\frac{1}{(-\beta)^n}\right| < \varepsilon.
\end{equation}
Let $\delta(\kappa):= \min\big\{ T_{-\beta}^n (1): n=1,2,\cdots, \kappa \big\}>0$. For any $x<\delta(\kappa)$, we have $N_x \cap \{1,2,\cdots, \kappa\} = \{1,2,\cdots, \kappa\}$. Hence, by \eqref{eq:h-n-x} and \eqref{eq:h-e-x}, we have
  $$
  \frac{\beta}{\beta+1}-\varepsilon <\frac{\beta}{\beta+1}-\sum_{n=\kappa+1}^{\f} \frac{1}{\beta^n}
  \leq    {h}_{-\beta}(x) =  \sum_{n \in N_x} \frac{1}{(-\beta)^n} \le \frac{\beta}{\beta+1} + \sum_{n=\kappa+1}^{\f} \frac{1}{\beta^n} = \frac{\beta}{\beta+1}+ \varepsilon,$$
  where $\frac{\beta}{\beta+1}=\sum_{n=0}^\infty \frac{1}{(-\beta)^n}$.
  Therefore, $$\lim_{x \to 0^{+}} {h}_{-\beta}(x) =\frac{\beta}{\beta+1}. $$

  \smallskip
For  (3), we first note that  $m\geq 2$ since $\beta$ is a non-integer real number. Take $\Delta(m):=\max\{T_{-\beta}^1(1),\cdots, T_{-\beta}^{m-1}(1)\}$. For any $x\in(\Delta(m),1]$, by \eqref{eq:h-n-x}, we have $N_x=\{km\}_{k\geq 0}$ and
$$h_{-\beta}(x)=\sum_{n \in N_x} \frac{1}{(-\beta)^n}=1+\frac{1}{(-\beta)^m}+\frac{1}{(-\beta)^{2m}}+\cdots
=\frac{\beta^m}{\beta^m-(-1)^m}.$$

\smallskip
For  (4), if $T_{-\beta}^n(1) < 1$ for all $n \geq 1$, a similar argument to the proof of  (2) shows that $\lim_{x \to 1^-} h_{-\beta}(x) = 1$.
\end{proof}

%
%
%
%
%
The following lemma will be used in our subsequent proof. We left the proof to the readers since it is an easy exercise in real analysis.
\begin{lemma}\label{lemma:limit}
  Let $f,g: (0,1) \to \R$ be two functions such that the set $\{x\in(0,1]:f(x) \neq  g(x)\}$ has no interior point.
  If the left limits $\displaystyle\lim_{x \to x_0^{-}} f(x)$ and $\displaystyle\lim_{x \to x_0^{-}} g(x)$ exist for some $0< x_0 \le 1$, then we have $$\lim_{x \to x_0^{-}} f(x) = \lim_{x \to x_0^{-}} g(x).$$
  Similarly, if the right limits $\displaystyle\lim_{x \to x_0^{+}} f(x)$ and $\displaystyle\lim_{x \to x_0^{+}} g(x)$ exist for some $0 \le x_0 < 1$, then $$\lim_{x \to x_0^{+}} f(x) = \lim_{x \to x_0^{+}} g(x).$$
\end{lemma}

We study the normalizing constants of the densities for  $\nu_{-\beta}$ by  Lemma \ref{lemma:limit}.

\begin{lemma}\label{le:equal}
 If $\beta_1$ and $\beta_2>1$ are two non-integer real numbers with $\nu_{-\beta_1}=\nu_{-\beta_2}$,  $K_{-\beta_1}$ and $ K_{-\beta_2}$,  the normalizing constants of the densities for $\nu_{-\beta_1}$ and $\nu_{-\beta_2}$ must satisfy that
  $$\frac{K_{-\beta_1}}{K_{-\beta_2}}=\frac{\beta_1(\beta_2+1)}{\beta_2(\beta_1+1)}.$$
\end{lemma}
\begin{proof}
It follows from Proposition \ref{prop:property} (2) and \eqref{eq:nor-density} that $$\lim_{x \to 0^+} \widetilde{h}_{-\beta_i}(x) = \frac{1}{K_{-\beta_i}} \cdot  \lim_{x \to 0^+} {h}_{-\beta_i}(x)= \frac{\beta_i}{K_{-\beta_i}(\beta_i+1)} \quad\text{for } i =1,2. $$
If $\nu_{-\beta_1} = \nu_{-\beta_2}$,
then, obviously, the densities $\widetilde{h}_{-\beta_1}(x) = \widetilde{h}_{-\beta_2}(x)$ for Lebesgue almost every $x \in (0,1]$.
By Lemma \ref{lemma:limit}, we obtain $$ \frac{\beta_1}{K_{-\beta_1}(\beta_1+1)} = \lim_{x \to 0^+} \widetilde{h}_{-\beta_1}(x) = \lim_{x \to 0^+} \widetilde{h}_{-\beta_2}(x) = \frac{\beta_2}{K_{-\beta_2}(\beta_2+1)}, $$
 yielding the required ratio $$\frac{K_{-\beta_1}}{K_{-\beta_2}}=\frac{\beta_1(\beta_2+1)}{\beta_2(\beta_1+1)}.$$
\end{proof}

From  Proposition \ref{prop:property} and \eqref{eq:nonor-den},
we see that $\{ T_{-\beta}^n (1): n \geq 1 \}$ plays a central role in studying the density of $\nu_{\beta}$. For any real number $\beta>1$, we denote  the orbit of \(1\) under \(T_{-\beta}\) by
\begin{equation}\label{eq:def-orb-1}
\mathcal{O}_{-\beta}:=\big\{ T_{-\beta}^n (1): n \geq 1 \big\}.
\end{equation}
Next, we prove that \(1\) lies in exactly one of $\mathcal{O}_{-\beta_1}$ and $\mathcal{O}_{-\beta_2}$ if \(\nu_{-\beta_1}=\nu_{-\beta_2}\) (where \(\beta_1,\beta_2>1\) are distinct non-integer real numbers).
At first, we show that \(1\) must appear in at least one of $\mathcal{O}_{-\beta_1}$ and $\mathcal{O}_{-\beta_2}$.

\begin{lemma}\label{lem:1_in_union}
Let \(\beta_1,\beta_2 > 1\) be two different non-integer real numbers.
If \(\nu_{-\beta_1}=\nu_{-\beta_2}\), then \(1 \in \mathcal{O}_{-\beta_1} \cup \mathcal{O}_{-\beta_2}\).
\end{lemma}

\begin{proof}
Suppose, by contradiction, that \(1 \notin \mathcal{O}_{-\beta_1} \cup \mathcal{O}_{-\beta_2}\).
By  Proposition \ref{prop:property} (4), we have
\[
\lim_{x \to 1^-} \widetilde{h}_{-\beta_i}(x)
= \frac{1}{K_{-\beta_i}} \cdot \lim_{x \to 1^-} h_{-\beta_i}(x)
= \frac{1}{K_{-\beta_i}} \quad \text{for } i =1,2.
\]
Since \(\nu_{-\beta_1}=\nu_{-\beta_2}\), we have \(\widetilde{h}_{-\beta_1}(x) = \widetilde{h}_{-\beta_2}(x)\) for almost all \(x\in(0,1]\). By Lemma~\ref{lemma:limit}, we obtain
\[
\frac{1}{K_{-\beta_1}}
= \lim_{x \to 1^-} \widetilde{h}_{-\beta_1}(x)
= \lim_{x \to 1^-} \widetilde{h}_{-\beta_2}(x)
= \frac{1}{K_{-\beta_2}},
\]
which implies \(K_{-\beta_1}=K_{-\beta_2}\). On the other hand, Lemma~\ref{le:equal} yields
\[
\frac{\beta_1(\beta_2+1)}{\beta_2(\beta_1+1)}
= \frac{K_{-\beta_1}}{K_{-\beta_2}}
= 1.
\]
So, \(\beta_1 = \beta_2\), which  contradicts with the assumption \(\beta_1\neq \beta_2\). Therefore, \(1 \in \mathcal{O}_{-\beta_1} \cup \mathcal{O}_{-\beta_2}\).
\end{proof}

The next result  shows that \(\mathcal{O}_{-\beta_1}\) and \(\mathcal{O}_{-\beta_2}\) are finite and equal after removing  \(1\).

\begin{proposition}\label{prop:finite}
Let \(\beta_1,\beta_2 > 1\) be two different non-integer real numbers.
If \(\nu_{-\beta_1}=\nu_{-\beta_2}\), then \(\mathcal{O}_{-\beta_1}\) and \(\mathcal{O}_{-\beta_2}\) are finite, and
$
\mathcal{O}_{-\beta_1} \setminus \{1\} = \mathcal{O}_{-\beta_2} \setminus \{1\}$.
\end{proposition}

\begin{proof}
By Lemma~\ref{lem:1_in_union}, we have $1 \in \mathcal{O}_{-\beta_1} \cup \mathcal{O}_{-\beta_2}$.
 Hence at least one of $\mathcal{O}_{-\beta_1}$ and $\mathcal{O}_{-\beta_2}$ contains $1$, and is therefore finite.
Without loss of generality, assume that $1 \in \mathcal{O}_{-\beta_1}$. Write
\[
\mathcal{O}_{-\beta_1} \cup \{0\} = \{ x_0, x_1, \cdots, x_{m} \},
\]
where \(0=x_0 < x_1 < \cdots  < x_{m}=1\).
For each \(0 \le k \le m-1\), since \((x_k, x_{k+1}) \cap \mathcal{O}_{-\beta_1} = \emptyset\), Proposition~\ref{prop:property} (1) implies that \(h_{-\beta_1}\) is constant on \((x_k, x_{k+1})\). Since \(\widetilde{h}_{-\beta_1} = \widetilde{h}_{-\beta_2}\) almost everywhere, Lemma~\ref{lemma:limit} shows that \(h_{-\beta_2}\) is also constant on \((x_k, x_{k+1})\). Applying Proposition~\ref{prop:property} (1) again, we deduce that \(\mathcal{O}_{-\beta_2} \cap (x_k, x_{k+1}) = \emptyset\) for every \(0 \le k \le m-1\). Therefore, \(\mathcal{O}_{-\beta_2} \setminus \{1\} \subset \{ x_1, \cdots, x_{m-1} \}= \mathcal{O}_{-\beta_1} \setminus \{1\}\), and  \(\mathcal{O}_{-\beta_2}\) is finite. By symmetry, the reverse inclusion \(\mathcal{O}_{-\beta_1} \setminus \{1\} \subset \mathcal{O}_{-\beta_2} \setminus \{1\}\) also holds. Then, we prove
$
\mathcal{O}_{-\beta_1} \setminus \{1\} = \mathcal{O}_{-\beta_2} \setminus \{1\}$.
\end{proof}

We now turn to prove that \(1\) lies in exactly one of $\mathcal{O}_{-\beta_1}$ and $\mathcal{O}_{-\beta_2}$ if \(\nu_{-\beta_1}=\nu_{-\beta_2}\) for two  distinct non-integer real numbers \(\beta_1,\beta_2>1\). Before proving it, it is worth highlighting the distinctions between the proof of the coincidence of invariant measure for  positive and negative \(\beta\)-transformations.
 In the positive \(\beta\) case, the coincidence of invariant measures implies that the normalizing constants of the densities for  $\nu_{\beta_1}$ and $\nu_{\beta_2}$ are equal;
 moreover, since \(0\) is a fixed point, evaluating the densities of $\nu_{\beta_1}$ and $\nu_{\beta_2}$ at \(0\) suffices to determine whether \(\nu_{\beta_1}=\nu_{\beta_2}\) \cite[Proposition 2.6]{Huang-Wang-2025}.
  In contrast, for the \((-\beta)\)-transformation, Lemma~\ref{le:equal} shows that the normalizing constants of the densities for  $\nu_{-\beta_1}$ and $\nu_{-\beta_2}$ satisfy $ K_{-\beta_1}/K_{-\beta_2}= {\beta_1(\beta_2+1)}/{\beta_2(\beta_1+1)}$. Furthermore,  \(0\) is not fixed. Thus, evaluating the densities for  $\nu_{-\beta_1}$ and $\nu_{-\beta_2}$ at \(0\) is not suffice  to prove $\nu_{-\beta_1}=\nu_{-\beta_2}$. Instead, we must examine every plateau of the step density function. This additional complexity is reflected in the following proposition and
also the proof of Theorem \ref{thm:measure}.

\begin{proposition}\label{prop:0}
  Let $\beta_1,\beta_2 > 1$ be two different non-integer real numbers. If $\nu_{-\beta_1}=\nu_{-\beta_2}$, then $1 \in \mathcal{O}_{-\beta_1} \setminus \mathcal{O}_{-\beta_2}$ or  $1 \in \mathcal{O}_{-\beta_2} \setminus \mathcal{O}_{-\beta_1}$ .
\end{proposition}
\begin{proof}
Since Lemma~\ref{lem:1_in_union} already rules out the case \(1 \notin \mathcal{O}_{-\beta_1} \cup \mathcal{O}_{-\beta_2}\), it remains to exclude  \(1 \in \mathcal{O}_{-\beta_1} \cap \mathcal{O}_{-\beta_2}\).
By contradiction,
assume that \(1 \in \mathcal{O}_{-\beta_1} \cap \mathcal{O}_{-\beta_2}\). Since \(\nu_{-\beta_1}=\nu_{-\beta_2}\), Proposition~\ref{prop:finite} yields
$
\mathcal{O}_{-\beta_1} = \mathcal{O}_{-\beta_2}$.
Hence,  $T_{-\beta_1}$ and $T_{-\beta_2}$ have the same minimal return time to \(1\). Denote the minimal return time by
\[
m := \min\{ n \ge 1 : T_{-\beta_1}^{n}(1) = 1 \}
   = \min\{ n \ge 1 : T_{-\beta_2}^{n}(1) = 1 \}.
\]
Then, write
\begin{equation*}
\mathcal{O}_{-\beta_1} \cup \{0\} = \mathcal{O}_{-\beta_2} \cup \{0\} = \{ x_0, x_1, \cdots, x_m \},
\end{equation*}
where \(0 = x_0 < x_1 < \cdots < x_m = 1\). Then, it follows from Proposition \ref{prop:property} (1) that $h_{-\beta_1}$ and $h_{-\beta_2}$ are step function.

\smallskip
We now compare each step of $h_{-\beta_1}$ and $h_{-\beta_2}$.
Since $\nu_{-\beta_1}=\nu_{-\beta_2}$, Lemma \ref{le:equal}  imply that
\begin{equation}\label{eq:density-ratio}
\frac{K_{-\beta_1}}{K_{-\beta_2}}=\frac{h_{-\beta_1}(x)}{h_{-\beta_2}(x)}
= \frac{\beta_1(\beta_2+1)}{\beta_2(\beta_1+1)}
\quad \text{for all } x \in (0,1].
\end{equation}
For any \(x \in (x_{m-1},1)\), Proposition~\ref{prop:property} (3) gives
\[
h_{-\beta_i}(x) = \frac{\beta_i^{m}}{\beta_i^{m}-(-1)^{m}}, \quad i=1,2.
\]
Combining this with \eqref{eq:density-ratio}, we obtain
\begin{equation*}
\frac{\beta_1(\beta_2+1)}{\beta_2(\beta_1+1)}=
\frac{K_{-\beta_1}}{K_{-\beta_2}}=\frac{h_{-\beta_1}}{h_{-\beta_2}}=
\frac{\beta_1^m(\beta_2^m-(-1)^m)}{\beta_2^m(\beta_1^m-(-1)^m)}
.
\end{equation*}
Without loss of generality, assume \(\beta_1 < \beta_2\). If \(m\) is even, then
\[
\frac{\beta_1(\beta_2+1)}{\beta_2(\beta_1+1)} < 1,
\quad \text{but} \quad
\frac{\beta_1^m(\beta_2^m-1)}{\beta_2^m(\beta_1^m-1)} > 1,
\]
which is a contradiction. Hence \(m\) must be odd, and we have
\begin{equation}\label{eq:k-k}
\frac{K_{-\beta_1}}{K_{-\beta_2}}=\frac{h_{-\beta_1}}{h_{-\beta_2}}
= \frac{\beta_1^m(\beta_2^m+1)}{\beta_2^m(\beta_1^m+1)}.
\end{equation}

Next, for each \(1 \le k \le m-1\), define
$n_k^i = \min\{ 1 \le n \le m-1 : T_{-\beta_i}^n(1) = x_k \}$, $i = 1,2$.
For \(x \in (x_{m-2}, x_{m-1})\), by \eqref{eq:h-n-x},  we have $N_x^i=\{n\geq 0:T_{-\beta_i}(1)\geq x\}=\{km:k\geq 0\}\cup \{km+n_{m-1}^i:k\geq 0\}$, and then,
\begin{equation*}
h_{-\beta_i}(x) =\left( 1 + \frac{1}{(-\beta_i)^{n_{m-1}^i}} \right) \sum_{k\geq 0}\frac{1}{(-\beta_i)^{km}}= \left( 1 + \frac{1}{(-\beta_i)^{n_{m-1}^i}} \right) \frac{\beta_i^m}{\beta_i^m+1}\quad\text{for } i=1,2.
\end{equation*}
Then we have
\begin{equation}\label{eq:h-h-m-1}
\frac{K_{-\beta_1}}{K_{-\beta_2}}=\frac{h_{-\beta_1}}{h_{-\beta_2}}
= \frac{1+\frac{1}{(-\beta_1)^{n_{m-1}^1}}}{1+\frac{1}{(-\beta_2)^{n_{m-1}^2}}}
\cdot\frac{\beta_1^m(\beta_2^m+1)}{\beta_2^m(\beta_1^m+1)}.
\end{equation}

Thus, by \eqref{eq:k-k} and \eqref{eq:h-h-m-1}, we obtain
\[ \frac{1+\frac{1}{(-\beta_1)^{n_{m-1}^1}}}{1+\frac{1}{(-\beta_2)^{n_{m-1}^2}}}
= 1,
\]
which implies
$
(-\beta_1)^{n_{m-1}^1} = (-\beta_2)^{n_{m-1}^2}$.
Continuing in the same way, we have
\[
\frac{1+\frac{1}{(-\beta_1)^{n_1^1}}}{1+\frac{1}{(-\beta_2)^{n_1^2}}}
= \cdots
= \frac{1+\frac{1}{(-\beta_1)^{n_{m-1}^1}}}{1+\frac{1}{(-\beta_2)^{n_{m-1}^2}}}
= 1,
\]
and hence
$
(-\beta_1)^{n_k^1} = (-\beta_2)^{n_k^2}
$ for all $1 \le k \le m-1$.
Therefore, since the exponents \(n_k^i\) are a permutation of the set \(\{1,\cdots,m-1\}\), we have
\[
\beta_1^{\frac{m(m-1)}{2}} =\beta_1^{\sum_{k=1}^{m-1} n_k^1} =\beta_2^{\sum_{k=1}^{m-1} n_k^2}=\beta_2^{\frac{m(m-1)}{2}},
\]
yielding \(\beta_1 = \beta_2\). This contradicts the assumption \(\beta_1 < \beta_2\). Thus \(1\) belongs to exactly one of the sets \(\mathcal{O}_{-\beta_1}\) and \(\mathcal{O}_{-\beta_2}\). \end{proof}

Now we are ready to prove our first main theorem.
\begin{proof}[Proof of Theorem \ref{thm:measure}]
The sufficiency follows from Proposition \ref{prop:suff}.
In the following, we focus on the necessity. Recall our hypothesis that $\beta_1, \beta_2 > 1$ are two distinct non-integer real numbers with $\nu_{-\beta_1} = \nu_{-\beta_2}$. By Proposition \ref{prop:0}, $1$ belongs to exactly one of $\mathcal{O}_{-\beta_1}$ and $\mathcal{O}_{-\beta_2}$. Assume that $1 \notin \mathcal{O}_{-\beta_1}$ and $1 \in \mathcal{O}_{-\beta_2}$.
Let $m := \min\{ n \ge 1: T_{-\beta_2}^n(1) = 1 \}$. By Proposition \ref{prop:finite}, we can  write \begin{equation}\label{eq:o-1-2}
\mathcal{O}_{-\beta_1} \cup \{1\} = \mathcal{O}_{-\beta_2} = \{x_0, \cdots, x_{m-1} \},\end{equation}
 where $x_i =T_{-\beta_2}^i(1)$ for $0\leq i\leq m-1$ and $x_0=T_{-\beta_2}^m(1) = 1$. We will show that $\widetilde{h}_{-\beta_1}(x)= \widetilde{h}_{-\beta_2}(x)$ almost everywhere only holds when $m=2$.

\smallskip
For $m=2$, we have $\mathcal{O}_{-\beta_1} =\{x_1\} $ and $ \mathcal{O}_{-\beta_2}=\{1,x_1\}$. It follows from Proposition \ref{prop:property} (3) and (4) that for any $x\in(x_{1},1)$,
\begin{equation*}
  {h}_{-\beta_1}(x) = 1 \quad \text{and} \quad
  {h}_{-\beta_2}(x) = \sum_{n=0}^{\infty} \frac{1}{(-\beta_2)^{2n}} = \frac{\beta_2^2}{\beta_2^2-1}.
\end{equation*}
By Lemma \ref{le:equal}, we obtain
$$   \frac{\beta_1(\beta_2+1)}{\beta_2(\beta_1+1)} =\frac{K_{-\beta_1}}{K_{-\beta_2}}=\frac{h_{-\beta_1}}{h_{-\beta_2}}= \frac{\beta_2^2 - 1}{\beta_2^2}, $$
which implies $\beta_2 = \beta_1 + 1$.

\smallskip
Furthermore, since $\mathcal{O}_{-\beta_1} =\{x_1\}$, we have $T_{-\beta_1}^n(1) = T_{-\beta_1}(1) = -\beta_1 + \lfloor\beta_1\rfloor + 1$ for all $n \geq 1$. Particularly, $T_{-\beta_1}^2(1) = T_{-\beta_1}(1)$, i.e.,
\begin{align*}
 -\beta_1 (-\beta_1 + \lfloor\beta_1\rfloor + 1)+\lfloor \beta_1 T_{-\beta_1}(1) \rfloor + 1=-\beta_1 + \lfloor\beta_1\rfloor + 1. \end{align*}
Then, a straightforward calculation reveals that
$$\beta_1^2-\lfloor\beta_1\rfloor \beta_1-(\lfloor\beta_1\rfloor-\lfloor \beta_1 T_{-\beta_1}(1) \rfloor)=0,$$
which implies that $\beta_1$ must be a root of the quadratic equation $x^2 - qx - p = 0$, where $p, q \in \mathbb{N}$ with $p \leq q$. Then we proved the necessity for $m=2$.

\smallskip
%
For $m \ge 3$, by \eqref{eq:nonor-den} and \eqref{eq:o-1-2}, we have that
\begin{equation}\label{eq:h-2}
\begin{split}
{h}_{-\beta_2}(x) =& \sum_{i=0}^{m-1}\left(\sum_{n\geq 1}\frac{\mathds{1}_{[0,x_i)}(x)}{(-\beta_2)^{nm+i}}\right)=
\left(\sum_{n\geq 1}\frac{1}{(-\beta_2)^{nm}}\right)\sum_{i=0}^{m-1}\frac{\mathds{1}_{[0,x_i)}(x)}{(-\beta_2)^i}
\\
= &
 \frac{\beta_2^m}{\beta_2^m-(-1)^m}\left( \sum_{i=0}^{m-1} \frac{ \mathds{1}_{[0,x_i)}(x)}{(-\beta_2)^i}  \right)
 \quad \text{for all } x\in(0,1].
 \end{split}
\end{equation}

\smallskip
For $\beta_1$,
write $y_k = T_{-\beta_1}^k (1)$ for $1 \leq k \leq m-1$.
By \eqref{eq:o-1-2}, we have
$$ \mathcal{O}_{-\beta_1} = \{ y_1, y_2, \cdots, y_{m-1} \} = \{ x_1, x_2, \cdots, x_{m-1} \}. $$
This implies that there exists a unique $\ell\in\{1,\cdots , m-1\}$ such that $T_{-\beta_1}^{m}(1) = y_{\ell}$.
Consequently, the sequence $\big\{ T_{-\beta_1}^n (1) \big\}_{n=1}^\infty$ is eventually periodic with period $y_\ell, \cdots, y_{m-1}$.
Using \eqref{eq:nonor-den} again, for all $x\in(0,1]$, we obtain
\begin{equation}\label{eq:beta-2}
\begin{split}
  &{h}_{-\beta_1}(x)\\
= & \mathds{1}_{[0,1)}(x) + \sum_{k=1}^{\ell-1} \frac{1}{(-\beta_1)^k} \cdot \mathds{1}_{[0,y_k)}(x) + \sum_{k=\ell}^{m-1} \bigg( \sum_{n=0}^{\infty} \frac{1}{(-\beta_1)^{n(m-\ell)+k}} \bigg) \cdot \mathds{1}_{[0,y_k)}(x) \\
= & \mathds{1}_{[0,1)}(x) + \sum_{k=1}^{\ell-1} \frac{1}{(-\beta_1)^k} \cdot \mathds{1}_{[0,y_k)}(x) + \frac{\beta_1^{m-\ell}}{\beta_1^{m-\ell} -(-1)^{m-\ell}} \left(\sum_{k=\ell}^{m-1}  \frac{1}{(-\beta_1)^{k}} \cdot \mathds{1}_{[0,y_k)}(x)\right).
\end{split}
\end{equation}

\smallskip
According to Lemma \ref{le:equal}, then, by comparing \eqref{eq:h-2} with \eqref{eq:beta-2}, for all $x\in(0,1]$, we have
\begin{equation*}
\begin{split}
  h_{-\beta_1}(x)=& \mathds{1}_{[0,1)}(x) + \sum_{k=1}^{\ell-1} \frac{1}{(-\beta_1)^k} \cdot \mathds{1}_{[0,y_k)}(x) + \frac{\beta_1^{m-\ell}}{\beta_1^{m-\ell} -(-1)^{m-\ell}} \sum_{k=\ell}^{m-1}  \frac{1}{(-\beta_1)^{k}} \cdot \mathds{1}_{[0,y_k)}(x)\\
= & \frac{K_{-\beta_1}}{K_{-\beta_2}}\cdot \frac{\beta_2^m}{\beta_2^m-(-1)^m}\left(\mathds{1}_{[0,1)}(x) + \sum_{k=1}^{m-1} \frac{1}{(-\beta_2)^k} \cdot \mathds{1}_{[0,x_k)}(x)\right)=\frac{K_{-\beta_1}}{K_{-\beta_2}}\cdot h_{-\beta_2}(x)
\end{split}
\end{equation*} 
So, we have $\frac{K_{-\beta_1}}{K_{-\beta_2}} \cdot\frac{\beta_2^m}{\beta_2^m-(-1)^m}=1$ by comparing the coefficients of $\mathds{1}_{[0,1)}$. And then, draw the conclusion that
$C = C_1 \cup C_2$,
where $$ C=\bigg\{ \frac{1}{(-\beta_2)^k}: 1 \le k \le m-1 \bigg\}, $$
and  $$ C_1= \bigg\{ \frac{1}{(-\beta_1)^k}: 1 \le k \le \ell-1 \bigg\}, \quad C_2=\bigg\{ \frac{\beta_1^{m-\ell}}{\beta_1^{m-\ell} -(-1)^{m-\ell}} \cdot \frac{1}{(-\beta_1)^k}: \ell \leq k \leq m-1 \bigg\}. $$

\smallskip
Denote by $M_C$ and $M_{C_1 \cup C_2}$
 the maximum absolute values of $C$ and $C_1 \cup C_2$, respectively; that is
 $$M_C=\max\{|x|:x\in C\}\quad \text{and}\quad  M_{C_1 \cup C_2}=\max\{|x|:x\in C_1 \cup C_2\}.$$
If $M_{C_1 \cup C_2}$ were attained in $C_1$, then, we have
$M_{C_1 \cup C_2} = 1/\beta_1 = 1/\beta_2 = M_C$,
which implies $\beta_1 = \beta_2$, contradicting the assumption $\beta_1 \neq \beta_2$.
Hence, $M_{C_1 \cup C_2}$ must be attained in $C_2$, yielding
\begin{equation}\label{eq:condition-2}
 M_C= \frac{1}{\beta_2} = \frac{\beta_1^{m-\ell}}{\beta_1^{m-\ell} - (-1)^{m-\ell}} \cdot \frac{1}{\beta_1^\ell}=M_{C_1 \cup C_2}.
\end{equation}

\smallskip
Next, consider the second-largest absolute value of $C$,  which is $1/\beta_2^2$. If this value were attained in $C_1$, it would imply
$1/\beta_2^2 = -1/\beta_1$, which is impossible since $\beta_1, \beta_2 > 1$.
Therefore, the second-largest absolute value of $C_1 \cup C_2$ must be in $C_2$. Consequently,
\begin{equation}\label{eq:condition-3}
  \frac{1}{\beta_2^2} = \frac{\beta_1^{m-\ell}}{\beta_1^{m-\ell} - (-1)^{m-\ell}} \cdot \frac{1}{\beta_1^{\ell+1}}.
\end{equation}
%
Comparing \eqref{eq:condition-2} and \eqref{eq:condition-3}, we have
\begin{equation*}
  \frac{1}{\beta_2^2} = \frac{\beta_1^{m-\ell}}{\beta_1^{m-\ell} -(-1)^{m-\ell}} \cdot \frac{1}{\beta_1^{\ell+1}} = \left(\frac{\beta_1^{m-\ell}}{\beta_1^{m-\ell} -(-1)^{m-\ell}} \cdot \frac{1}{\beta_1^\ell}\right)^2,
\end{equation*}
which implies that $\frac{1}{\beta_1}=\frac{\beta_1^{m-\ell}}{\beta_1^{m-\ell} -(-1)^{m-\ell}} \cdot \frac{1}{\beta_1^\ell}>0$. Substituting this back into \eqref{eq:condition-2} gives $1/\beta_1 = 1/\beta_2$, and thus $\beta_1=\beta_2$, which contradicts to $\beta_1 \neq \beta_2$. Therefore, the case $m \ge 3$ is impossible.
%
\end{proof}

\begin{remark}\label{rem:dynamical-insight}
The proof of Theorem \ref{thm:measure} reveals that the step density functions exactly have two steps, if $\nu_{-\beta_1} = \nu_{-\beta_2}$ for two distinct non-integer real numbers $\beta_1,\beta_2>1$. By Proposition \ref{prop:0}, we have \(1\) lies in exactly one of $\mathcal{O}_{-\beta_1}$ and $\mathcal{O}_{-\beta_2}$.
Then, for one parameter, the orbit of $1$ must reach a fixed point at the fastest possible rate,  i.e., $1 \to x_1\, \rotatebox[origin=c]{90}{$\circlearrowleft$}\,$. While for the other parameter, the orbit of $1$ must form a  pure 2-cycle, i.e., $1^{\,\curvearrowleft}_{\raisebox{2.0ex}{\,\rotatebox[origin=c]{180}{$\scriptstyle\curvearrowleft$}}} \, x_1$.
\end{remark}

\section{Matching property and sofic}\label{sec:4}
Recall that we say  $T_{-\beta}$ satisfying the matching property if there exists a finite integer $n \ge 1$ such that $T_{-\beta}^{n}(0) = T_{-\beta}^{n}(1)$. In this section, we will study the link between  the matching property and Yrrap numbers and
prove Theorem \ref{thm:matching}.
Note that $T_{-\beta}(0) = 1$. Then, $T_{-\beta}^{i-1}(1) = T_{-\beta}^i(0)$ for all $i \ge 1$, which allows us to characterize the matching property via $\{T_{-\beta}^n(1)\}_{n \ge 0}$.
\begin{proposition} \label{le:orbit-fixed}
For any real number $\beta > 1$, $T_{-\beta}$ satisfies the matching property if and only if the sequence $\{T_{-\beta}^n(1)\}_{n \ge 0}$ eventually lands on a fixed point of $T_{-\beta}$.
\end{proposition}
\begin{proof}
If $T_{-\beta}$ satisfies the matching property, then, there exists a finite integer $n \ge 1$ such that $T_{-\beta}^n(1) = T_{-\beta}^n(0)$. By $T_{-\beta}^n(0) = T_{-\beta}^{n-1}(1)$, we obtain $T_{-\beta}^n(1) = T_{-\beta}^{n-1}(1)$, which implies $T_{-\beta}^{i}(1) = T_{-\beta}^{n-1}(1)$ for all $i \ge n - 1$. Hence, $\{T_{-\beta}^n(1)\}_{n \ge 0}$ eventually reaches a fixed point.

\smallskip
Conversely, if $\{T_{-\beta}^n(1)\}_{n \ge 0}$ eventually lands on a fixed point, then, there exists $n \ge 1$ such that $T_{-\beta}^n(1) = T_{-\beta}^{n-1}(1)$. Since $T_{-\beta}^{n-1}(1) = T_{-\beta}^n(0)$, it follows immediately that $T_{-\beta}^n(1) = T_{-\beta}^n(0)$, which proves that $T_{-\beta}$ satisfies the matching property.
\end{proof}

Recall that for any integers $q\geq 1$ and $ m\geq 2$,  the generalized multinacci number $\beta_{q,m} \in (q, q+1)$  is  the unique real root of  $ x^m= q  x^{m-1}+\cdots +qx+ q$. We now prove that all generalized multinacci numbers satisfy  the matching property.

\begin{theorem}\label{prorpo1}
If $\beta$ is a generalized multinacci number, then $T_{-\beta}$ has matching property.
\end{theorem}
\begin{proof}
For any integers $q\geq 1$ and $ m\geq 2$, we have $\beta^m = q\beta^{m-1} + q\beta^{m-2} + \cdots + q$. Then,
dividing both sides by $\beta^m$ and $\beta^{m-1}$, respectively, we obtain
\begin{equation}\label{eq1}
1 = q \cdot \beta^{-1} + q \cdot \beta^{-2} + \cdots + q \cdot \beta^{-m}\quad \text{and}\quad \beta = q +q\cdot \beta^{-1} + q \cdot \beta^{-2} + \cdots + q \cdot \beta^{-m+1}.
\end{equation}
By Proposition \ref{le:orbit-fixed}, we only need to prove that there exists an $n $ such that $T_{-\beta}^{n+1}(1) = T_{-\beta}^n(1)$.

\smallskip
We claim that for $1\leq k\leq m-1$,
\begin{equation}\label{eq:T-beta-q-m-1}
T_{-\beta}^k(1) =
\begin{cases}
 \sum\limits_{i=0}^{\frac{k-1}{2}} \frac{q}{\beta^{m-2i}}=\frac{q}{\beta^m}+
\frac{q}{\beta^{m-2}}+\cdots+\frac{q}{\beta^{m-k+1}},
\quad \text{if } k\text{ is odd;}\\
\sum\limits_{i=0}^{\frac{k}{2} } \frac{q}{\beta^{m-2i}}+\sum\limits_{i=1}^{m-k-1}\frac{q}{\beta^{i}}= \\
\quad \quad \quad \quad \quad \frac{q}{\beta^m}+\frac{q}{\beta^{m-2}}+\cdots+\frac{q}{\beta^{m-k}}+\frac{q}{\beta^{m-k-1}}
+\cdots+\frac{q}{\beta},\, \ \text{if} \ k\text{ is even.}
\end{cases}
\end{equation}
By convention, we set $\sum_{i=1}^0\frac{q}{\beta^i}=0$.
 We now prove the claim by computing $\{T_{-\beta}^k(1)\}_{0\leq k \leq m-1}$ inductively.

\smallskip
For $k=1$, by $q < \beta < q+1$, we have
\[
T_{-\beta}^1(1) = -\beta \cdot 1 + \lfloor \beta \cdot 1 \rfloor + 1 = -\beta + q + 1=\frac{q}{\beta^m},
\]
where the last equation follows by \eqref{eq1}. Thus, \eqref{eq:T-beta-q-m-1} holds for $k=1$.

\smallskip
For $k=2$, since $0<\beta \cdot T_{-\beta}^1(1) = \frac{q}{\beta^{m-1}}<1$, we obtain
\[
T_{-\beta}^2(1) = -\beta \cdot T_{-\beta}^1(1) + 0 + 1 = -\frac{q}{\beta^{m-1}} + 1=\sum_{i=1}^{m-2}\frac{q}{\beta^i}+\frac{q}{\beta^m},
\]
where that last equation follows by \eqref{eq1}.

\smallskip
Suppose that \eqref{eq:T-beta-q-m-1} holds for an even integer $k$ with $1\leq k\leq m-2$, i.e.,
$$T_{-\beta}^k(1)=\sum\limits_{i=0}^{\frac{k}{2} } \frac{q}{\beta^{m-2i}}+\sum\limits_{i=1}^{m-k-1}\frac{q}{\beta^{i}}=
\frac{q}{\beta^m}+\cdots+\frac{q}{\beta^{m-k}}+\frac{q}{\beta^{m-k-1}}
+\cdots+\frac{q}{\beta}.$$
Then, by \eqref{eq1}, a straightforward calculation implies that
$$T_{-\beta}^{k+1}(1)=-\beta T_{-\beta}^k(1) +q+1=\sum\limits_{i=0}^{\frac{k+1-1}{2} } \frac{q}{\beta^{m-2i}}=
\frac{q}{\beta^m}+\cdots+\frac{q}{\beta^{m-k}}.$$
Similarly, if \eqref{eq:T-beta-q-m-1} holds for an odd integer $k$ with $1\leq k\leq m-2$, i.e.,
$$T_{-\beta}^k(1) =
\sum\limits_{i=0}^{\frac{k-1}{2}} \frac{q}{\beta^{m-2i}}=\frac{q}{\beta^m}+
\frac{q}{\beta^{m-2}}+\cdots+\frac{q}{\beta^{m-k+1}}.$$
Then,
\begin{align*}
T_{-\beta}^{k+1}(1)=-\beta T_{-\beta}^k(1) +1=&\sum\limits_{i=0}^{\frac{k+1}{2} } \frac{q}{\beta^{m-2i}}+\sum_{i=1}^{k-2}\frac{q}{\beta^i}\\
=&
\frac{q}{\beta^m}+\cdots+\frac{q}{\beta^{m-k-1}}+
\frac{q}{\beta^{m-k-2}}+\cdots+\frac{q}{\beta}.
\end{align*}
So, we establish  \eqref{eq:T-beta-q-m-1} for $k+1$, proving the claim.

\smallskip
Now, for even $m$, applying \eqref{eq1} and \eqref{eq:T-beta-q-m-1}, a simple calculation shows that
$$T_{-\beta}^m(1) = -\beta T_{-\beta}^{m-1}(1) + 1 = \frac{q}{\beta^m} + \frac{q}{\beta^{m-2}} + \cdots + \frac{q}{\beta^2} = T_{-\beta}^{m-1}(1).$$
Similarly, we can verify that $T_{-\beta}^m(1) = T_{-\beta}^{m-1}(1)$ when $m$ is odd. Thus, we proved that $T_{-\beta}$ satisfies the matching property for the generalized multinacci number $\beta$.
\end{proof}

\begin{remark}\label{re}
The converse of Proposition \ref{prorpo1} is not true.  For instance, let $\beta = 1.324717 \cdots \in (1,2)$ with $d_{-\beta}(1) = 2112^\infty$, then, $T_{-\beta}$ satisfies the matching property. However, $\beta$ is not a generalized multinacci number.
\end{remark}

Proposition \ref{le:orbit-fixed} proves that $T_{-\beta}$ satisfying the matching property is equivalent to $\{T_{-\beta}^k(1)\}_{k\geq 1}$ being eventually periodic with period length $1$. Moreover,   Proposition \ref{lem:SFT-condition} (1) shows that $(-\beta)$-shifts being sofic is equivalent to that  $\{T_{-\beta}^n(1)\}_{n\geq 1}$ is eventually periodic with any period length.
To describe all sofic $(-\beta)$-shifts, we generalize the matching condition. We say that  $T_{-\beta}$ satisfies the \emph{general matching property} if there exist integers $m, n \in \mathbb{N}$ such that
$$T_{-\beta}^m(1) = T_{-\beta}^n(0),\qquad m\ne n-1.$$
As mentioned in the introduction, we call $\beta$  a Yrrap number if the corresponding $(-\beta)$-shifts is sofic. Then, general matching property connects the concepts of Yrrap numbers and sofic subshifts. We summarize these equivalent conditions in the following proposition.

\begin{proposition} \label{thm:general-matching-sofic}
For any real number $\beta > 1$, the followings are equivalent:
\begin{enumerate}
    \item $T_{-\beta}$ satisfies the general matching property;
    \item $\beta$ is a Yrrap number;
    \item the $(-\beta)$-expansion of $1$ is eventually periodic;
    \item the $(-\beta)$-shift $(\mathcal{S}_{-\beta},\sigma)$ is sofic.
\end{enumerate}
\end{proposition}

\section{Density of simple $(-\beta)$ numbers}\label{sec:5}
Recall that  $\beta > 1$ is called a simple $(-\beta)$-number if its $(-\beta)$-expansion of $1$ $d_{-\beta}(1)$ is periodic. In this section, we study the density of simple $(-\beta)$ numbers and prove Theorem \ref{thm:sft-matching}. Our strategy for proving Theorem \ref{thm:sft-matching} is to construct a sequence of periodic $(-\beta)$-expansions of 1, which can be  arbitrarily close to any given  $(-\beta)$-expansion of $1$.
If $\beta$ is a simple $(-\beta)$-number, i.e.,
the $(-\beta)$-expansion of $1$ is periodic,  the result is obvious. Thus, we only  consider those $\beta$ with the $(-\beta)$-expansion of $1$ is non-periodic.

\smallskip
For positive $\beta$-expansions, simple $\beta$-numbers are dense \cite[Theorem 5]{Parry-1960}, as one can approximate any $\beta$-expansion of $1$ via simple $n$-term truncations.
 In the negative $\beta$-case, however, this straightforward truncation fails, and constructing periodic approximations involves some difficulties. Note that a sequence is the $(-\beta)$-expansion of $1$ for some real number $\beta > 1$ if and only if it satisfies  the  conditions in Proposition \ref{prop:beta-expansion-1}.
 Verifying the conditions in Proposition \ref{prop:beta-expansion-1}
 needs the alternating lexicographical order of $(-\beta)$-shift, which depends on the parity of the first differing index.
 Consequently, constructing well-defined periodic approximations requires a parity-based analysis.  We exclude the forbidden patterns by using the Lyndon--Sch\"{u}tzenberger Theorem \cite{Lyndon-Schutzenberger-1962}, which is stated in the following lemma.

 %
%

\begin{lemma}\label{le:ls}
Given two words $\mathbf{a},\mathbf{b}\in\mathcal{A}^*$, if $\mathbf{ba}=\mathbf{ab}$, then, there exist a word $\mathbf{w}\in\mathcal{A}^*$ and integers $p, q \ge 1$ such that $\mathbf{a} = \mathbf{w}^p$ and $\mathbf{b} = \mathbf{w}^q$.
\end{lemma}

\smallskip
We say that an infinite sequence $d_1 d_2 \cdots \in \mathcal{A}^\mathbb{N}$ is \emph{self-admissible} if $d_k d_{k+1} \cdots \preceq d_1 d_2 \cdots$ for all $k \geq 1$. The following proposition proves that any non-periodic $(-\beta)$-expansion of $1$ can be approximated  by a self-admissible periodic sequence.
\begin{proposition}\label{le:contract-periodic}
If $\beta>1$ and the $(-\beta)$-expansion of 1 is non-periodic, then there exists a
 self-admissible periodic sequence $\{\upsilon_{\beta,n}\}_{n\geq 1}$ such that $\lim_{n\to\infty}\rho
(d_{-\beta}(1), \upsilon_{\beta,n}) =0.$
\end{proposition}
\begin{proof}
We construct the self-admissible periodic  sequence by choosing appropriate prefixes of $d_{-\beta}(1)=d_1d_2\cdots$. Note that for any sequence $(c_i)$ with $c_1 = \lfloor \beta \rfloor + 1$, it is easy to check that  $\sigma^k(c_i) \prec (c_i)$ if $c_{k+1} < \lfloor\beta\rfloor + 1$. Then, we divide the proof into two cases depending on  $d_{-\beta}(1)$ containing finitely many occurrences of $\lfloor \beta \rfloor + 1$ or not.

\smallskip
\noindent\textbf{Case 1:} If $d_{-\beta}(1)$ contains finitely many occurrences of $\lfloor \beta \rfloor + 1$,  for any  $ n \geq 1$, let
\begin{equation}\label{eq:N-v-n-1}
\upsilon_{\beta,n} := (d_1 \cdots d_{2N} d_{2N+1} \cdots d_{2N+n})^\infty,  \quad \text{where } N =N(\beta):= \max\{k : d_{k+1} = \lfloor \beta \rfloor + 1\}.
\end{equation}
We now prove  $\sigma^{i}(\upsilon_{\beta,n}) \preceq \upsilon_{\beta,n}$ for all $i \geq 1$. Since $\upsilon_{\beta,n}$ is periodic, we only need to consider $1 \leq i \leq 2N+n-1$. For $1 \leq i \leq N$, we have $d_{i+1} \cdots d_{2N+n} \prec d_{-\beta}(1)|_{2N+n-i}=d_1\cdots d_{2N+n-i}$ since
 the self-admissibility of $d_{-\beta}(1)$ and $d_{i+1} \cdots d_{2N+n} $ containing less $\lfloor \beta \rfloor + 1$ than $d_1\cdots d_{2N+n-i}$. So, $\sigma^i(\upsilon_{\beta,n}) \prec \upsilon_{\beta,n}$ for $1 \leq i \leq N$.
For $N+1\leq i \leq 2N+n-1 $, by \eqref{eq:N-v-n-1}, we have $ d_{i+1}<\lfloor \beta \rfloor + 1$. Then, $d_{i+1}\cdots  d_{2N+n} \prec d_{-\beta}(1)|_{2N+n-i}=d_1\cdots d_{2N+n-i}$, which implies $\sigma^i(\upsilon_{\beta,n}) \prec \upsilon_{\beta,n}$. Thus, we proved $\sigma^i(\upsilon_{\beta,n}) \preceq \upsilon_{\beta,n}$ for all $1 \leq i \leq 2N+n-1$.

\smallskip
\noindent\textbf{Case 2:} If $d_{-\beta}(1) = d_1 d_2\cdots$ contains infinitely many occurrences of $\lfloor \beta \rfloor + 1$, there exists a strictly increasing integer sequence $\{k_n\}_{n\geq 1}$ with $k_1 \geq 3$ such that $ d_{k_n} = \lfloor \beta \rfloor + 1$. Let
\begin{equation}\label{eq:d-j-n}
j=j(k_n):=\min\{1\leq i\leq k_n-1:d_{i+1}\cdots d_{k_n} = d_{1}\cdots d_{k_n-i}\},
\end{equation}
where the existence of $j$  is guaranteed by $ d_1 =  d_{k_n} = \lfloor \beta \rfloor + 1$.
Moreover, we denote by
\begin{equation}\label{eq:d-j-n-k}
l=l(n):=\max\{i \geq k_n: d_{j+1}\cdots d_{k_n}\cdots  d_{i} = d_{1}\cdots d_{k_n-j}\cdots d_{i-j}\}.
\end{equation}
Since $d_{-\beta}(1)$ is non-periodic and self-admissible, we have $l<\infty$ and  $\sigma^j(d_{-\beta}(1)) \prec d_{-\beta}(1)$. By \eqref{eq:d-j-n-k}, we have $ d_{l-j+1} \neq  d_{l+1}$.
However, by the alternating order \eqref{eq:alter-order},  the relation between $d_{l+1}$ and $d_{l-j+1}$ depends on the parity of $l-j+1$.

\smallskip
 If $l-j$ is odd, then  $\sigma^j(d_{-\beta}(1)) \prec d_{-\beta}(1)$ and \eqref{eq:alter-order} imply that  $ d_{l-j+1} <  d_{l+1}$. Set
\begin{equation}\label{eq:v-2}
\upsilon_{\beta,n}: = ( d_1 d_2\cdots d_l d_{l+1})^\infty.
\end{equation}
It is obvious that  $\sigma^i(\upsilon_{\beta,n}) \prec \upsilon_{\beta,n}$ if $d_i<\lfloor \beta \rfloor + 1$ and $1\leq i\leq l$. So, to prove the self-admissibility of $\upsilon_{\beta,n}$, we only need to  check those  indices $1 \leq i \leq l$ with $d_{i+1} = \lfloor \beta \rfloor + 1$. For $1 \leq i < j$, the self-admissibility of $d_{-\beta}(1)$ and \eqref{eq:d-j-n} yield $d_{i+1}\cdots d_{k_n} \prec d_1\cdots d_{{k_n}-i}$, showing $\sigma^i(\upsilon_{\beta,n}) \prec \upsilon_{\beta,n}$.

 \smallskip
 Next, we consider $j \leq i < l+1$. If $l-i$ is odd, \eqref{eq:d-j-n-k} gives $d_{i+1}\cdots d_l = d_{i-j+1}\cdots d_{l-j}$. Since $d_{l-j+1} < d_{l+1}$, we have
 $d_{i+1}\cdots d_l d_{l+1} \prec d_{i-j+1}\cdots d_{l-j} d_{l-j+1} \preceq d_1\cdots d_{l-i} d_{l-i+1}$,
 implying $\sigma^i(\upsilon_{\beta,n}) \prec \upsilon_{\beta,n}$.

 \smallskip
 If $l-i$ is even, to prove  $\sigma^i(\upsilon_{\beta,n}) \prec \upsilon_{\beta,n}$, we only need to show
 \begin{equation}\label{eq222}
 (d_{i+1} \cdots d_{l+1})(d_{1} \cdots d_{i}) \prec (d_{1} \cdots d_{l-i+1})(d_{l-i+2} \cdots d_{l+1}).
 \end{equation}
 The self-admissibility of $d_{-\beta}(1)$ yields the following two inequalities:
 \begin{align*}
 &d_{i+1}\cdots d_{l+1} \preceq d_1\cdots d_{l-i+1}, \quad \text{since } \sigma^i(d_{-\beta}(1))
  \preceq d_{-\beta}(1);\\
  &d_{l-i+2} \cdots d_{l+1} \preceq d_1 \cdots d_i, \quad \text{since } \sigma^{l-i+1}(d_{-\beta}(1)) \preceq d_{-\beta}(1).
  \end{align*}
  Since $l-i+1$ is odd,  combining the above two inequalities yields that
  \begin{equation}\label{eqeq1}
      (d_{i+1} \cdots d_{l+1})(d_{1} \cdots d_{i}) \preceq (d_{1} \cdots d_{l-i+1})(d_{l-i+2} \cdots d_{l+1}).
  \end{equation}
  If the $\prec$ holds, by \eqref{eq222}, we proved that $\sigma^i(\upsilon_{\beta,n}) \prec \upsilon_{\beta,n}$.

  \smallskip
  Otherwise, we have
  $
   (d_{1} \cdots d_{l-i+1})(d_{l-i+2} \cdots d_{l+1})=(d_{i+1} \cdots d_{l+1})(d_{1} \cdots d_{i})$.
Then, by Lemma \ref{le:ls}, there exist a word $\mathbf{w}$ and positive integers $p,q$ such that $d_1 \cdots d_i=\mathbf{w}^{p}$ and $ d_{i+1} \cdots d_{l+1}=\mathbf{w}^{q}$.
Since $|\mathbf{b}| = l-i+1 = q|\mathbf{w}|$ is odd, then $|\mathbf{w}|:=t$ must be odd.
   By self-admissibility of $d_{-\beta}(1)= \mathbf{w}^{p+q} d_{l+2}d_{l+3}\cdots$, we have
   \begin{equation}\label{eq:w-t}
   \begin{split}
   \sigma^t(d_{-\beta}(1)) &= \mathbf{w}^{p+q-1} d_{l+2}d_{l+3}\cdots \preceq \mathbf{w}^{p+q} d_{l+2}d_{l+3}\cdots\quad \text{and }\\ \sigma^{2t}(d_{-\beta}(1)) &= \mathbf{w}^{p+q-2} d_{l+2}d_{l+3}\cdots \preceq \mathbf{w}^{p+q} d_{l+2}d_{l+3}\cdots.\end{split}\end{equation}
    Since $t$ is odd, $(p+q-1)t$ and $(p+q-2)t$ have opposite parities.
     So, by \eqref{eq:alter-order} and \eqref{eq:w-t} against $\mathbf{w}$, we have  $d_{l+2}\cdots d_{l+t+1} \preceq \mathbf{w}$ and $d_{l+2}\cdots d_{l+t+1}\succeq \mathbf{w}$. Thus, $d_{l+2}\cdots d_{l+t+1} = \mathbf{w}$.
   Applying this argument inductively proves that $d_{-\beta}(1) = \mathbf{w}^\infty$,  contradicting to that $d_{-\beta}(1)$ is non-periodic. Thus, the equality in \eqref{eqeq1} cannot hold. So, we proved   $\sigma^i(\upsilon_{\beta,n}) \prec \upsilon_{\beta,n}$.

\smallskip
  If $l-j$ is even, we define
 \begin{equation}\label{eq:v-222}
 \upsilon_{\beta,n} := (d_{1}\cdots d_{l} d_{l+1} d_{l+2})^\infty  . \end{equation}
 The  symmetric  arguments for $l-j$ being odd can prove $\sigma^{i}(\upsilon_{\beta,n}) \prec \upsilon_{\beta,n}$ for all $1 \leq i < l+2$.

 \smallskip
 The constructions in \eqref{eq:N-v-n-1}, \eqref{eq:v-2} and \eqref{eq:v-222} give a self-admissible periodic sequences $\{\upsilon_{\beta,n}\}$, which can share arbitrarily long prefixes with $d_{-\beta}(1)$ for $n$ large enough. Then, by \eqref{eq:rho}, we have  $\rho
(d_{-\beta}(1), \upsilon_{\beta,n})$ converges to $0$ as $n\to\infty$.
 \end{proof}

Note that we construct the  approximations $\{\upsilon_{\beta,n}\}$ from two directions. If the period of  $\upsilon_{\beta,n}$ is odd, then $\upsilon_{\beta,n} \prec d_{-\beta}(1)$; if the period is even, then $\upsilon_{\beta,n} \succ d_{-\beta}(1)$.
To illustrate the construction in Proposition \ref{le:contract-periodic}, we give the following two examples.
\begin{enumerate}
\item Let $\beta \in (1,2)$ be the real root of $\beta^4 - \beta^3 - \beta^2 - \beta - 1 = 0$. Then, $d_{-\beta}(1) = 2121^\infty$. Since the symbol $2$ appears finitely many times, by \eqref{eq:N-v-n-1}, we have  $N = 2$ and then,  $\upsilon_{\beta,n} = (21211^n)^\infty$ for $n \geq 1$.
\item Let $\beta \in (1,2)$ satisfy $d_{-\beta}(1) = 21212211(21212212)^\infty$, containing infinitely many $2$s. By \eqref{eq:v-2}, let $\upsilon_{\beta,n} = \bigl(21212211(21212212)^n\bigr)^\infty$ for $n \geq 1$.
\end{enumerate}

\smallskip
Now we are ready to prove Theorem \ref{thm:sft-matching}.

\begin{proof}[Proof of Theorem \ref{thm:sft-matching}]
Let $\beta>1$ with non-periodic $d_{-\beta}(1)=d_1d_2\cdots$. Our strategy for proving Theorem \ref{thm:sft-matching} is to show that the periodic sequence $\{\upsilon_{\beta,n}\}$, constructed in Proposition \ref{le:contract-periodic} 
satisfy the conditions in Proposition \ref{prop:beta-expansion-1} for $n$ large enough. Then, there exist a unique real number $\beta_n>1$ such that  $d_{-\beta_n}(1)=\upsilon_{\beta,n}$ for $n$ large enough and  $\lim_{n \to \infty} |\beta - \beta_n| = 0$.

\smallskip
We first prove that the sequences $\{\upsilon_{\beta,n}\}$ satisfy the conditions in Proposition \ref{prop:beta-expansion-1}. It follows from Proposition \ref{le:contract-periodic} that $\upsilon_{\beta,n}$ satisfy the first conditions in Proposition \ref{prop:beta-expansion-1} for all $n\geq 1$.
Since $d_{-\beta}(1)$ satisfies the second condition in Proposition \ref{prop:beta-expansion-1}, i.e.,
$d_{-\beta}(1) \succ w_1 w_2 \cdots = 211222112112112221122\cdots$,
then, let $M = M(\beta) := \min\{i : d_i \neq w_i\}$. Hence, we can choose large enough $n$, such that $\upsilon_{\beta,n}=(d_1\cdots d_{k_n})^\infty$ with
$k_n > M$. Then  $\upsilon_{\beta,n} \succ w_1 w_2 \cdots$. For all $n\geq 1$, $\upsilon_{\beta,n}$ naturally avoids the forbidden forms in Proposition \ref{prop:beta-expansion-1} (3) and (4), since $\upsilon_{\beta,n}$ is periodic. Thus, we proved  the sequence $\upsilon_{\beta,n}$ satisfy the conditions in Proposition \ref{prop:beta-expansion-1} $n$ large enough.

\smallskip
 By Proposition \ref{prop:beta-expansion-1},  there exist a unique real number $\beta_n >1$ with $d_{-\beta_n}(1) = \upsilon_{\beta,n}=(d_1\cdots d_{k_n})^\infty$ for $n$ large enough. It follows from   Proposition \ref{prop:entropy} and \eqref{eq:entropy} that
 $$h_{\rm{top}}(\mathcal{S}_{-\beta_n}) = \lim_{i \to \infty} \frac{1}{i} \log|\mathcal{L}_i(\mathcal{S}_{-\beta_n})| = \log \beta_n.$$
By \eqref{eq:S-beta-odd} and \eqref{eq:S-beta-non-odd}, we have $\mathcal{L}_i(\mathcal{S}_{-\beta_n}) = \mathcal{L}_i(\mathcal{S}_{-\beta})$ for all $i \leq k_n$. 
Thus, we have $$\lim_{n\to\infty} \log \beta_n=\lim_{n\to\infty} h_{\rm{top}}(\mathcal{S}_{-\beta_n})= \lim_{n\to\infty}\lim_{i \to \infty} \frac{1}{i} \log|\mathcal{L}_i(\mathcal{S}_{-\beta_n})|=\lim_{i \to \infty} \frac{1}{i} \log|\mathcal{L}_i(\mathcal{S}_{-\beta})|
=\log \beta,$$
 which proves  $\lim_{n \to \infty} \beta_n = \beta$. Hence, the set of simple $(-\beta)$-numbers is dense in $(1,+\infty)$.
\end{proof}

\section*{Acknowledgement}
The authors are grateful to Hajime Kaneko for bringing
the problem of the coincidence of the  invariant measures for $(-\beta)$-transformations
 to our attention, and Wolfgang Steiner for many helpful discussions.
  We also thank  Zhiqiang Wang for pointing out the
  reference \cite{Brown-Yin-1999}.
  The third author is supported by the Fundamental Research Funds
   for the Central Universities (No.3120626135) and NSFC (No.
12601157).

\end{document}